\documentclass[a4paper,12pt]{elsarticle}

\usepackage[T1]{fontenc}
\usepackage{microtype}

\usepackage{cmlgc}
\usepackage{ucs}

\usepackage{color}

\usepackage{amssymb,amsfonts,amsmath,amsthm}
\usepackage[all]{xy}
\usepackage[colorlinks=true]{hyperref}
\usepackage{mathrsfs}
\usepackage{graphicx}
\usepackage{rotating}
\usepackage{array}

\usepackage{float}
\usepackage{xspace}

\usepackage[english]{babel}

\newcommand{\bI}{\mathbf{I}}

\newcommand{\Be}{\mathbf{e}}
\newcommand{\Bf}{\mathbf{f}}

\newcommand{\Bh}{\mathbf{h}}

\newcommand{\Bq}{\mathbf{q}}
\newcommand{\Bt}{\mathbf{t}}

\newcommand{\Bx}{\mathbf{x}}
\newcommand{\By}{\mathbf{y}}
\newcommand{\Bz}{\mathbf{z}}

\newcommand{\Beps}{{\boldsymbol \epsilon}}
\newcommand{\Bxi}{{\boldsymbol \xi}}
\newcommand{\Btheta}{{\boldsymbol \theta}}

\newcommand{\cB}{\mathcal{B}}
\newcommand{\cC}{\mathcal{C}}

\newcommand{\cF}{\mathcal{F}}
\newcommand{\cH}{\mathcal{H}}
\newcommand{\cI}{\mathcal{I}}

\newcommand{\cL}{\mathcal{L}}

\newcommand{\cO}{\mathcal{O}}
\newcommand{\cP}{\mathcal{P}}
\newcommand{\cQ}{\mathcal{Q}}

\newcommand{\cS}{\mathcal{S}}
\newcommand{\cT}{\mathcal{T}}
\newcommand{\cU}{\mathcal{U}}
\newcommand{\cV}{\mathcal{V}}
\newcommand{\cW}{\mathcal{W}}

\newcommand{\ue}{\textup{e}}
\newcommand{\uf}{\textup{f}}
\newcommand{\uh}{\textup{h}}
\newcommand{\ux}{\textup{x}}

\newcommand{\ke}{k_{\textup{e}}}
\newcommand{\kf}{k_{\textup{f}}}
\newcommand{\kh}{k_{\textup{h}}}
\newcommand{\kx}{k_{\textup{x}}}

\newcommand{\act}{{\mathbb{A}}} 
\newcommand{\NN}{{\mathbb{N}}}  
\newcommand{\RR}{{\mathbb{R}}}  
\newcommand{\TT}{{\mathbb{T}}}  
\newcommand{\ZZ}{{\mathbb{Z}}}  

\newcommand{\rotxc}[1]{\begin{sideways}#1\end{sideways}}

\newcommand{\rotyc}[1]{\rotxc{\rotxc{\rotxc{#1}}}}

\newcommand{\acts}{ \text{ \rotyc{$\circlearrowright$} } } 

\renewcommand{\act}{Action-Angle\xspace}
\newcommand{\neigh}{neighborhood\xspace}

\newcommand{\SVN}{San V\~u Ng\d{o}c\xspace}

\newcommand{\smooth}{{\cC^\infty}}   
\renewcommand{\flat}{\mathcal{F}\! \ell^{\infty}}

\newcommand{\Diag}{{\operatorname{Diag}}} 

\newcommand{\pois}[2]{{\{#1,#2\}}}  

\newcommand{\hypref}[2]{{\hyperref[#1]{#2~\ref{#1}}}}

\newcommand{\ifwork}[1]{\ifthenelse{\boolean{workmode}}{#1}{}}
\newcommand{\comment}[1]{}
\newcommand{\mute}[1]{}
\newcommand{\printname}[1]{}

    \newtheorem{theorem}{Theorem}[section]

\newtheorem{proposition}[theorem]{Proposition}

\newtheorem{lemma}[theorem]{Lemma}

\theoremstyle{plain}
\newtheorem{definition}[theorem]{Definition}

\newtheorem{remark}[theorem]{Remark}

\newtheorem{assumption}[theorem]{Assumption}

\newtheorem{notations}[theorem]{Notations}

\newcommand{\abs}[1]{\left|#1\right|}

\newcommand{\re}{\mathop{\mathfrak{R}e}\nolimits}
\newcommand{\im}{\mathop{\mathfrak{I}m}\nolimits}

\newcommand{\flechebas}[1]{%
  \settoheight{\unitlength}{\mbox{$#1$}}%
  \settowidth{\Taille}{\mbox{~${\scriptstyle #1}$}}%
  \addtolength{\unitlength}{4ex}%
  \begin{picture}(0,1)
    \put(0,1){\vector(0,-1){1}}
    \put(0,0.5){\makebox(0,0){${\scriptstyle #1}$ \hspace{\the\Taille}}}
  \end{picture}}
\newcommand{\flechehaut}[1]{%
  \settoheight{\unitlength}{\mbox{$#1$}}%
  \settowidth{\Taille}{\mbox{~${\scriptstyle #1}$}}%
  \addtolength{\unitlength}{4ex}%
  \begin{picture}(0,1)
    \put(0,0){\vector(0,1){1}}
    \put(0,0.5){\makebox(0,0){\hspace{\the\Taille}${\scriptstyle #1}$ }}
  \end{picture}}
\newcommand{\flechedroite}[1]{%
  \settowidth{\unitlength}{\mbox{$#1$}}
  \settoheight{\Taille}{\mbox{${\scriptstyle #1}$}}
  \addtolength{\Taille}{1ex}
  \addtolength{\unitlength}{4ex}
  \raisebox{0.5ex}{%
  \begin{picture}(1,0)
    \put(0,0){\vector(1,0){1}}
    \put(0.5,0){\makebox(0,0){${\scriptstyle #1}$ \vspace{\the\Taille}}}
  \end{picture}}}
\newcommand{\flechegauche}[1]{%
  \settowidth{\unitlength}{\mbox{$#1$}}
  \settoheight{\Taille}{\mbox{${\scriptstyle #1}$}}
  \addtolength{\Taille}{1ex}
  \addtolength{\unitlength}{4ex}
  \begin{picture}(1,0)
    \put(1,0){\vector(-1,0){1}}
    \put(0.5,0){\makebox(0,0){${\scriptstyle #1}$ \vspace{\the\Taille}}}
  \end{picture}}

\newcommand{\ouf}{\vspace{8mm}}

\newcommand{\AGS}{Atiyah - Guillemin \& Sternberg\xspace}

\title{Local models of almost-toric integrable systems: theory and applications}
\author{Christophe Wacheux\footnote{Assistant professor at EPFL, christophe.wacheux@epfl.ch}}

\begin{document}

\maketitle

\begin{abstract}
In this article we show how one can use the local models of integrable Hamiltonian systems near critical points to give a description of certain singular loci of integrables semi-toric systems in dimension $\geqslant 4$.  
\end{abstract}

\section{Introduction}
 Given a symplectic manifold $(M^{2n},\omega)$, we define an intergrable Hamiltonian system (or IHS) as a function $F = (f_1,\ldots,f_n):M \to \RR^n$ such that its components commute for the Poisson $\pois{.}{.}$ bracket induced by $\omega$, and the set of points for which $dF$ is of maximal rank is an open dense subset (such points are said \emph{regular} points, and \emph{critical} otherwise; in particular, $p \in M$ is a \emph{fixed} point if $dF(p)= 0$). For the associated foliation $\cF$ of $F$ given by the connected components of its fibers, we set $\pi_\cF : M \to \cB$, and $\cB$ the base space.
 
 The description of IHS is a great and very difficult task, and a classification of IHS would be an important step in that direction. For a subclass of IHS and its associated equivalence relation, a classification describes the structure of the moduli space and a basis of simple and universal objects for the qualitative and quantitative analysis of a given system.
 
 As IHS take their origins in mechanical systems, it is natural to consider classifications that involve, for instance, periodic motions, particular trajectories (e.g. fixed points), and the behaviour of the system near these trajectories. We will consider here a subclass of IHS defined by the periodicity of the flow of the components of $F$, and on the nature of its critical points.
 
 \subsection{Almost-toric, semi-toric and toric systems}
 
 We first remind the theory of non-degenerate critical points for IHS. Given an IHS $(M,\omega,F)$, at a fixed point $p$, the set of the Hessians of the $f_i$'s at $p$ denoted $\cH[f_i]_p$ form a vector space $\langle \cH[F]_p \rangle$. The Poisson bracket $\pois{.}{.}$ induces a Poisson bracket $\pois{.}{.}_p$ on $\langle \cH[F]_p \rangle$, turning it into a Lie subalgebra of $\mathfrak{sp}(2n)$. The fixed point $p$ is called \emph{non-degenerate} if $\langle \cH[F]_p \rangle$ is a Cartan subalgebra of $\mathfrak{sp}(2n)$, that is, if it is abelian and self-centralizing. 
 
 Now, let us assume that $p$ is a general critical point and $\cU_p$ a neighborhood of $p$, one can take the symplectic quotient $(\tilde{\cU}_p = \cU_p /{\! \! \slash} \TT^\kx,\tilde{\omega}_p)$ of $\cU_p$ by the $\kx$-torus action induced by regular components $F_\ux$ of $F$. The restriction $\tilde{F}_p$ of $F$ is an IHS for $(\tilde{\cU}_p = \cU_p /{\! \! \slash} \TT^\kx,\tilde{\omega}_p)$, and the projection $\tilde{p}$ of $p$ on $\tilde{\cU}_p$ is a fixed point for $\tilde{F}_p$. The critical point $p$ is called nondegenerate if $\tilde{p}$ is a nondegenerate fixed point for $\tilde{F}_p$.
 
 The classification of Cartan subalgebra of $\mathfrak{sp}(2n)$ due to Williamson~\cite{Williamson-OnAlgPbLinNF-1936} give the following result concerning the quadratic part of a nondegenerate critical point $p$. Through all the paper, we use the bold notations $\Bx_l = (x_1,\ldots,x_l)$, omitting the $l$ when the context makes it obvious. For instance, we shall note here the Darboux coordinates of $M$ by $(\Bx,\Bxi) = (x_1,\ldots,x_n,\xi_1,\ldots,\xi_n)$.
 
 
 \begin{theorem}[Williamson, 1936] \label{theo:Williamson}
 
 Let $p \in M$ a nondegenerate critical point of $F$ an IHS. Then there exists a quadruplet $\Bbbk=(\ke,\kf,\kh,\kx) \in \NN^4$, an open set $\cU_p \supseteq p$ and a symplectomorphism $\varphi: (\cU_p,\omega,p) \to (\RR^{2n},\omega_0 = \sum_{i=1}^n d\xi_i \wedge dx_i,0)$ such that \[ \varphi^*F = Q_\Bbbk + o(2), \text{ with } Q_\Bbbk = (\Be_{\ke},\Bh_{\kh},\Bf_{\kf},\xi_{n-\kx+1},\ldots,\xi_n) \text{ and}\]
 
 \begin{itemize}
  \item $\ue_i = x^2_i + \xi^2_i$ - elliptic (or $E$) components,
  \item $\uh_i = x_i y_i$ - hyperbolic (or $H$) components,
  \item $f_i =(f^1_i,f^2_i)$, $\begin{cases}
         \uf^1_i = x^1_i \xi^1_i + x^2_i \xi^2_i \\
         \uf^2_i = x^1_i \xi^2_i - x^2_i \xi^1_i
        \end{cases}$- focus-focus (or $FF$) components.
 \end{itemize}
 
\end{theorem}
 
 The quadruplet $\Bbbk$ is called the Williamson type of $p$ and it is a symplectic invariant. Note that Theorem~\ref{theo:Williamson} can also apply to regular points with $\Bbbk = (0,0,0,n)$. The four coefficients are linked by the equation

\begin{equation} \label{equ:Williamson_type}
 \ke + 2\kf + \kh + \kx = n
\end{equation}

\begin{definition}
We define $\cW(F)$ as the set of different Williamson types that occurs for a given IHS $F$. When equipped with the following relation \[ \Bbbk \preccurlyeq \Bbbk' \text{ if: } \ke \geqslant \ke', \kf \geqslant \kf' \text{ and } \kh \geqslant \kh' ,\] it is a (partially) ordered set (the term \emph{poset} also appears in the litterature). 

\end{definition}

Let us show that $(\cW(F),\preccurlyeq)$ is an ordered set. Let $\Bbbk, \Bbbk', \Bbbk'' \in \cW(F)$.

\begin{itemize}
 \item {\bf reflexivity:} we always have $\ke \geqslant \ke, \kf \geqslant \kf$, and $\kh \geqslant \kh$, thus $\Bbbk \preccurlyeq \Bbbk $,
 \item {\bf antisymetry:} if $\Bbbk \preccurlyeq \Bbbk'$ and $\Bbbk' \preccurlyeq \Bbbk$, then $ \ke \geqslant \ke'$, $\kf \geqslant \kf'$, $\kh \geqslant \kh'$ and $ \ke \leqslant \ke'$, $\kf \leqslant \kf'$, $\kh \leqslant \kh'$, so $ \ke = \ke'$, $\kf = \kf'$, $\kh = \kh'$ and hence, by equation~\ref{equ:Williamson_type} we have $\kx = \kx'$, so $\Bbbk = \Bbbk'$,
 \item {\bf transitivity:} if $\Bbbk \preccurlyeq \Bbbk'$ and $\Bbbk' \preccurlyeq \Bbbk''$ then $ \ke \geqslant \ke' \geqslant \ke'', \kf \geqslant \kf' \geqslant \kf'', \kh \geqslant \kh' \geqslant \kh''$, hence $\Bbbk \preccurlyeq \Bbbk''$.
 \end{itemize}

 We introduce the following notations
 
\begin{notations}
 \begin{itemize} 
 \item $P_\Bbbk (\cU) - $ The locus of critical points of Williamson type $\Bbbk$ on the open set $\cU \subseteq M$.
 \item $L_\Bbbk (\cU) - $ The locus of critical leaves $\{ \Lambda_b | b \in \cU \}$ of Williamson type $\Bbbk$, with $\cU$ an open set of $M$.
 \item $V_\Bbbk (\cU):= F(P_\Bbbk (\cU)) - $ The locus of critical values of Williamson type $\Bbbk$ on the open set $\cU \subseteq M$.
  \end{itemize}
\end{notations}

 
 We can now define the almost-toric systems. We introduce a criterium called \emph{complexity}. This notion find its origins in the works of Karshon and Tolman~\cite{KarshonTolman-Centeredcomplexity1HamTorusAction-2001}\cite{KarshonTolman-CompleteinvariantsforHamT_actions_tall-2003}\cite{KarshonTolman-ClassificationofHamiltonian-2011}, Symington and Leung~\cite{Symington-4from2-2001}, \cite{LeungSymington-Almosttoricsymplectic-2010}, and of \SVN in~\cite{SVN-Momentpolytopessymplectic-2007}.

\begin{definition}
 Let $F=(f_1,\ldots,f_n):(M^{2n},\omega) \to \RR^n$ be an integrable Hamiltonian system. It is said to be almost-toric of complexity $c \leqslant n$ if it verifies these conditions:
 
 \begin{itemize}
  \item all critical points are non-degenerate,
  \item there are no singularities of hyperbolic type: $k_\uh = 0$,
  \item the flow generated by each of the last ($n-c$) components of $F$ function is $2\pi$-periodic so that $\check{F}^{c}:=(f_{c+1},\ldots,f_n)$ generates a Hamiltonian $\TT^{n-c}$-action.
  \end{itemize}

 If $c=0$, the system is called \emph{toric}, and \emph{semi-toric} if $c=1$. For semi-toric systems, we set $f_1$ as the function that may fail to yield an $S^1$-action, and define $\check{F}:=(f_2,\ldots,f_n)$ as the function that generates the $\TT^{n-1}$-action.
 
\end{definition} 
 
 For toric systems, a very simple and powerful classification has been achieved. For these systems, we have the two following results
 
 \begin{theorem}
 \begin{itemize}
  \item {\bf \AGS theorem:} If $(M,\omega, F)$ is equipped with a Hamiltonian $\TT^r$-action, then the fibers of its associated moment map are connected and its image is a rational convex polytope of dimension $r$ (\cite{Atiyah-ConvexityandCommuting-1982}, \cite{GuilleminSternberg-ConvexitypropertiesI-1982}, \cite{GuilleminSternberg-ConvexitypropertiesII-1984}).
  
  \item {\bf Delzant's classification theorem:} In the integrable case $r=n$, if the action is effective, $F(M)$ characterizes the IHS up to a symplectomorphism equivariant with respect to the $\TT^n$ action generated by $F$ (\cite{Delzant-Hamiltoniensperiodiqueset-1988}, \cite{Delzant-ClassificationActions-1990}).
  
 \end{itemize}
  
 \end{theorem}

 There are several possible directions for a classification ``\`a la Delzant'' that extends the toric case: replace $\TT^n$ by any (non-abelian) compact Lie group, have $F$ non-necessarily proper etc. The almost-toric extension is of a different nature. In toric systems, there are $n$ \emph{globally} Hamiltonian $S^1$-actions, and (hence) only elliptic critical points occurs. In almost-toric systems, we require fewer $S^1$-actions, and we authorize critical points with elliptic and focus-focus components. Hence, to our consideration, semi-toric systems are the closest from toric systems. The ransom of this generalization is the loss of the rigidity on which we relied for toric systems: almost-toric systems cover more physical situations (see~\cite{PelayoRatiuSVN-SymplecticBifTheoryForIntegrableSystems-2011} and references therein), but the image is not a moment polytope anymore, and Delzant theorem do not apply anymore.
 
 Nevertheless, the image of the moment map still contains a lot of information. In~\cite{SVN-semiglobalinvariants-2003}, \cite{PelayoSVN-Semitoricintegrablesystems-2009} and~\cite{PelayoSVN-ConstructingIntSysOfSemitoricType-2011}, Pelayo and \SVN retrieved a classification ``{\`a} la Delzant'' for semi-toric systems of dimension $2n=4$. This classification requires a description of the image of the moment map and of the $V_\Bbbk(M)$. The aim of this article is to provide results for this description in any dimension, when $\kf = 1$.

\subsection{Localization of semi-toric critical values} \label{subsection:Loc_S-T_crit_values}

This subsection is devoted to the formulation and explanation of our main result. Let $\vec{e}_1 $ be the first vector of the basis induced by $F$ - $\vec{e}_1 $ is the direction of $f_1$.

\begin{theorem} \label{theo:loc_FF_values}
 Let $F$ be a semi-toric integrable system on a compact symplectic manifold $M^{2n}$, $\Bbbk \in \cW(F)$ with $\kf = 1$. Then $V_\Bbbk(M)$ is a finite union of connected embedded submanifolds $\Gamma_i$ of dimension $\kx$ called nodal submanifolds such that:
\begin{enumerate}
 
 \item For each $\Gamma_i$, there exists an affine plane $\cP(\Gamma_i) \subseteq \RR^n$ of the form $\cP(\Gamma_i) = P + \RR \cdot \vec{e}_1 + \RR \cdot \vec{v}_1 + \ldots + \RR \cdot \vec{v}_{\kx}$ with $P \in F(M)$ and $\vec{v}_1,\ldots,\vec{v}_{\kx}$ a free family with integer coefficients, such that $\Gamma_i \subseteq \cP(\Gamma_i) \cap F(M)$.
 
 \item In $\cP(\Gamma_i)$, the nodal surface is the graph of a smooth function $h$ from an open affine domain $D \subseteq \RR^{\kx}$ to $\RR$:
 
 \[ \Gamma_i = \{ P + h({\bf t}) \cdot \vec{e}_1 + t_1 \cdot \vec{v}_1 + \cdots + t_{\kx} \cdot \vec{v}_{\kx}\ ,\  {\bf t} \in D \} .\]

  \item If we assume that the fibers are connected, then the nodal surfaces are isolated: there exists a open neighborhood $\cV_i$ of $\Gamma_i$ such that $V_\Bbbk(\cV_i) = \Gamma_i$.

\end{enumerate}
\end{theorem}


\begin{figure}
 \centering
  \includegraphics[scale=0.6]{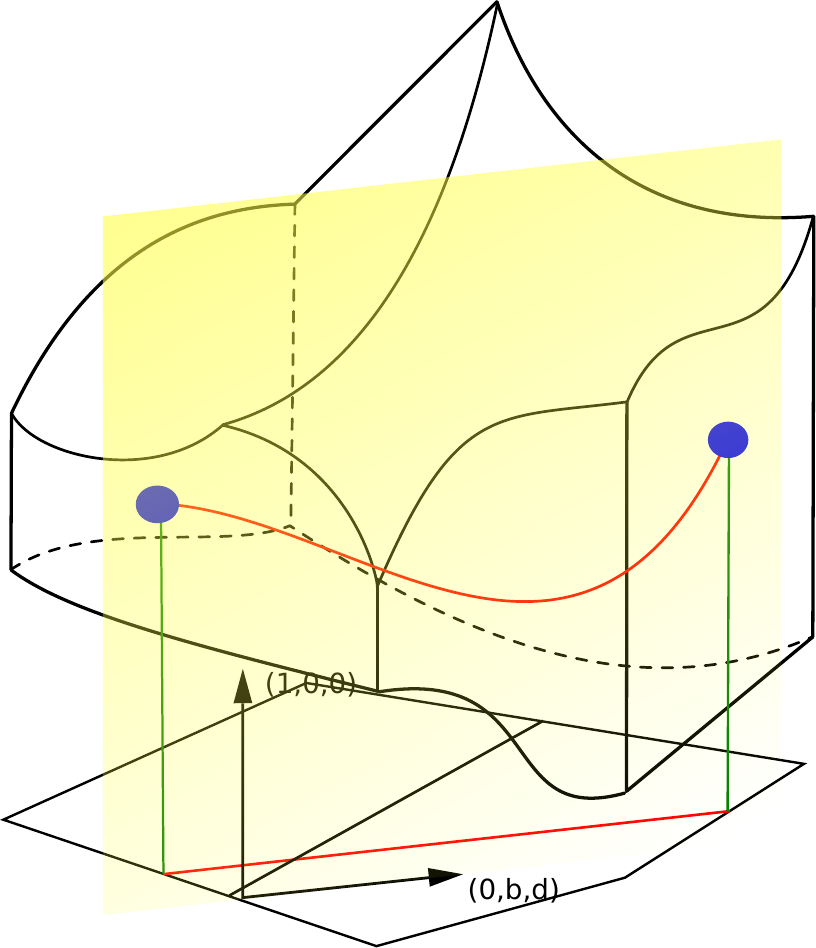}
 \caption{A nodal path of $FF-X$ critical values}
\end{figure}

In particular, this theorem answers negatively to a question asked to me by Colin de Verdi\`ere in 2010: ``Can we have a ``loop'' of focus-focus-transverse singular values in dimension $2n=6$ ?''. We must thank him deeply for this simple question that acted both as a compass and as an incentive in my research during the years 2010-2011. We developped the techniques of local models in particular to answer to this question.


In the theorem above, note that we speak of the image of critical points, and not of critical value, as the Williamson type of a fiber may not be well defined. As a result, a value can belong to different $V_{\Bbbk}(M)$'s. Yet, we chose to give a result describing the image of the moment map rather than the base space of the foliation, because our ``local model'' results describe the former. Description of the latter requires the introduction of new structures and a study on its own. This is what we actually work on in a paper with N.T. Zung that is to be published in the following months, hopefully. Another reason is that the image of the moment map is the space that physicists directly have access to by experimentation, especially when dealing with their quantum countepart.

Organization of this paper is the following: first we set up background and remind existing results concerning the existence of normal forms for the foliation, and thus for ``local models'' for the image of moment map. Then we give precisions for the image of a semi-toric critical point, and we finish with the proof of the global result Theorem~\ref{theo:loc_FF_values} using the results about local models proved in section~\ref{section:loc_mod_ST_val}.
 
\section{Normal forms for points, orbits and leaves}

In this section, we remind the existing results of normal forms near critical points. By normal form theorem, we mean a decomposition of a function into ``simple'' elements with respect to a fixed equivalence relation. We start with Eliasson normal form and its orbital extension with Miranda-Zung's theorem. We finish with Zung's leaf-wise theorem, the ``Arnold-Liouville theorem with singularities''.

In this paper, we will use the following notations

\begin{notations}
 We set the convention, given row vectors $\Bx$ and $\By$ of the same size and $z$ a single coordinate:
\begin{equation} \label{equ:wedge_bold}
d\Bx \wedge d\By = \sum_{j=1}^r dx_j \wedge dy_j \hspace{5mm} \text{ and } \hspace{5mm} d\Bx \wedge dz = \sum_{j=1}^r dx_j \wedge dz
\end{equation}

Lastly, for $A,B \in M_{p,q} (\RR)$, $A \bullet B = (a_{ij} b_{ij})_{i,j=1..n}$. 
 
\end{notations}

\subsection{Eliasson normal form}

In 1984, Eliasson proved the following theorems although he only published the first one. We state the theorem and give to the reader the necessary references for a full proof: 

\begin{theorem}[Eliasson Normal Form - Semi-toric case] \label{theo:Eliasson_NF-ST_case}
 Let $(M^{2n},\omega,F)$ a semi-toric integrable system with $n \geqslant 2$, and $m$ a critical point of Williamson type $\Bbbk = FF - E^{k_\ue} - X^{\kx}$.
 
 Then there exists a triplet $(\cU_m,\varphi,G_\Bbbk)$ with $\cU_m$ an open neighborhood of $m$ in $M$, $\varphi_{\Bbbk}$ a symplectomorphism of $\cU_m$ to a neighborhood of $(0 \in \RR^{2n} , \omega_0)$ and $G_{\Bbbk}$ a local diffeomorphism of $0 \in \RR^n$ to itself such that:
 
 \[ \varphi^*_{\Bbbk} F = G_{\Bbbk} \circ Q_{\Bbbk} \]
  
\end{theorem}

It was the contribution of many people that allowed eventually the statement of the theorems above. The first works to be cited here are those of Birkhoff, Vey~\cite{Vey-SurCertainsSystemes-1978}, Colin de Verdière and Vey~\cite{ColinVey-LemmeMorseIsochore-1979}, and of course Eliasson in~\cite{Eliasson-Thesis-1984} and~\cite{Eliasson-NormalformsHamiltonian-1990}. More recently, Chaperon in~\cite{Chaperon-GeoDiff-SingSysDyn-Asterisque-1986} and~\cite{Chaperon-SmoothFFaSimpleProof-2012}, Zung in \cite{Zung-Anoteonfocusfocus-1997} and~\cite{Zung-Anothernotefocus-2002}, and \SVN \& Wacheux in~\cite{SVNWacheux-SmoothNF_for_IHS_near_ff_sing-2013} provided new proofs and filled the technical gaps that remained in the original proof.
 
\subsection{Semi-local normal form}

Eliasson normal form is the first of many results generalizing the symplectic linearization of integrable systems.

Let $F:(M^{2n},\omega) \to \RR^n$ be a proper integrable semi-toric system. The orbit $\cO_m$ of a critical point $m \in M$ by the local Poisson $\RR^n$-action is a submanifold of dimension equal to the rank~$\kx$ of the action at the point $m$. For this section, we can assume without loss of generality that $df_1 \wedge \ldots \wedge df_{\kx} \neq 0$, that is, they are the transverse components of the critical point.

\begin{definition}
 The orbit $\cO_m$ is called non-degenerate if, when we take the symplectic quotient of a neighborhood of $\cO_m$ by the Poisson action of $\RR^{\kx}$ generated by $F_X:= (f_1,\ldots,f_{\kx})$, the image of $m$ is a non-degenerate fixed point.
\end{definition}

A non-degenerate orbit has only non-degenerate critical points of the same Williamson index. Thus it makes sense to talk of an orbit of a given Williamson index. The linear model of a non-degenerate orbit shall be for us the linear model of any point of this orbit. Of course, a non-degenerate Hamiltonian system has only non-degenerate orbits and non-degenerate leaves.

\begin{theorem}[Miranda \& Zung,~\cite{MirandaZung-Equivariantnormalform-2004}] 
\label{theo:M-Z}
Let $m$ be a critical point of Williamson type $\Bbbk$ of a semi-toric system $(M,\omega,F)$.

Then there exists a neighborhood~$\tilde{\cU}_m$ saturated with respect to the action of $F_X$ the transverse components of $F$ and a symplectomorphism: \[\varphi: (\tilde{\cU}_m ,\omega) \to \varphi(\tilde{\cU}_m) \subset L_\Bbbk \] such that:

\begin{itemize}
 \item[$\bullet$] $\varphi^\star F = Q_\Bbbk $,
 \item[$\bullet$] The transverse orbit $\cO_{F_X} (m)$ is sent to the zero-torus \[ \cT = \{ \Bx^{e,f} = \Bxi^{e,f} = 0 \ ,\ \bI = 0 \} \] of dimension $\kx$.
\end{itemize}

Moreover, if there exists a symplectic action of a compact group $G \acts M$ that preserves the moment map $F$, the action can be linearized equivariantly with respect to that group action.

\end{theorem}

This theorem is an {\bf extension} of Eliasson normal form: it means that we can linearize the singular Lagrangian foliation of an integrable semi-toric system symplectically on an \emph{orbital} neighborhood of a non-degenerate critical point. The normal form is valid in a neighborhood that is larger than a ``point-wise'' $\varepsilon$-ball of a critical point: the \neigh is saturated with respect to the transverse action of the system. 

\subsection{Arnold-Liouville with singularities} \label{subsection:A-L_with_sing}

 We designate by $\cF:= \{ \text{connected components of } F^{-1}(c) | c \in Im(F) \} $ the foliation associated to $F$, and $\cB$ its base space. A leaf hence is a connected component of a fiber. We have the diagram
 
 \[
\xymatrix{ M \ar[rd]_{F}  \ar[r]_{\pi_\cF} &  \cB \ar[d]^{\tilde{F}} \\ 
\; & \hspace{7mm} F(M) \subset \RR^n }
\]

\begin{definition}

 A critical leaf of $\cF$ is a leaf in $\cF$ that contains a critical point for $F$. A singularity of $F$ is defined as (a germ of) a tubular neighborhood of a critical leaf.
 
\end{definition}

\subsubsection{Stratification by the orbits}

One of the consequence of the existence of focus-focus and hyperbolic critical points is that there is a distinction between the leaf containing a point and the joint orbit through that point. The following proposition describes precisely how non-degenerate critical leaves are stratified by orbits of different Williamson types:

\begin{proposition} \label{prop:closure_orbit}
 Let $m \in M$ be a point of Williamson type $\Bbbk$ of a proper, non-degenerate integrable Hamiltonian system $F$. Then:
 
\begin{enumerate}
 \item $\cO_F (m)$ is diffeomorphic to a direct product $\TT^c \times \RR^o$ (and $c+o = \kx$).
 \item For any point $m'$ in the closure of $\cO_F (m)$, $k_{\ue }(m') = k_{\ue }$, $k_{\uh}(m') \geqslant k_{\uh}$ and $k_{\uf}(m') \geqslant k_{\uf}$.
 \item The quantities $k_{\ue }$, $k_{\uf} + c$ and $k_{\uf} + o$ are invariants \emph{of the leaf}.
 \item For a non-degenerate proper semi-toric system, a leaf $\Lambda$ contains a finite number of $F$-orbits with a minimal $\Bbbk$ for $\preccurlyeq$, and the Williamson type for these $F$-orbits is the same. 
\end{enumerate}

\end{proposition}

All these assertions are proven by Zung in~\cite{Zung-SymplecticTopology_I-1996}. The last statement asserts that, in a non-degenerate critical leaf, a point with minimal Williamson type is not unique in general, but the minimal Williamson type is. 

This allows us to give the following definition:

\begin{definition}
 For a non-degenerate semi-toric system $(M,\omega,F)$, the Williamson type of a leaf $\Lambda$ is defined as 
\[ \Bbbk(\Lambda):= \min_{\preccurlyeq} \{ \Bbbk(m) \ |\  m \in \Lambda \} \ .\]

\end{definition}

That quantity is well defined because the only critical leaves that difers from those of toric systems, are the ones that contains singularities with focus-focus components. Since leaves are closed, points with minimal Williamson type must contain one $FF$ component and the maximal number of $E$ components. That ensures us the unicity of the Williamson type of the leaf. Let us mention here that while we have defined the Williamson type of non-degenerate orbits and leaves, there is no relevant notion of a Williamson type of a fiber.

\begin{definition}
 A non-degenerate critical leaf $\Lambda$ is called \emph{topologically stable} if there exists a saturated neighborhood $\cV(\Lambda)$ and a $\cU \subset \cV(\Lambda)$ a small neighborhood of a point $m$ {\bf of minimal rank} in $\Lambda$ such that 
\[ \forall \Bbbk \in \cW^n_0 \ ,\ F(P_\Bbbk (\cV(\Lambda)))  = F(P_\Bbbk (\cU)) \ . \]

 An integrable system will be called \emph{topologically stable} if all its critical points are non-degenerate and topologically stable.
\end{definition}

The assumption of topological stability rules out some pathological behaviours that can occur for general foliations. Note however that for all known examples, the non-degenerate critical leaves are all topologically stable, and it is conjectured that it is also the case for all analytic systems. 

Since the papers~\cite{Zung-SymplecticTopology_I-1996} and~\cite{Zung-SymplecticTopology_II-2003} of Zung, the terminology concerning the assumption of topological stability has evolved. One speaks now of the \emph{transversality assumption}, or the \emph{non-splitting condition}. This terminology was proposed first by Bolsinov and Fomenko in~\cite{BolsinovFomenko-book}.

We set in the next definition both the equivalence relation and the basis of simple ``basis'' for a leaf-wise normal form:

\begin{definition}
We say that two singularities are isomorphic if they are leaf-wise homeomorphic. We name the following singularities isomorphism classes ``simple'':

\begin{itemize}
 \item[$\bullet$]  A singularity is called of \emph{(simple) elliptic type} if it is isomorphic to $\cL^\ue$: a plane $\RR^2$ foliated by $q_e$.
 \item[$\bullet$]  A singularity is called of \emph{(simple) focus-focus type} if it is isomorphic to $\cL^f$, where $\cL^f$ is given by $\RR^4$ locally foliated by $q_1$ and $q_2$. One can show (see Proposition $6.2$ in~\cite{SVN-BohrSommerfeld-2000}) that the focus-focus critical leaf must be homeomorphic to a pinched torus $\check{\TT}^2$: a 2-sphere with two points identified. The regular leaves around are regular tori.

\end{itemize}
\end{definition}

Properties of elliptic and focus-focus singularities are discussed in details in~\cite{Zung-SymplecticTopology_I-1996}. In particular, the fact that we can extend the Hamiltonian $S^1$-action that exists near a focus-focus point to a tubular neighborhood of the focus-focus singularity guarantees that the simple focus-focus singularity is indeed a pinched torus.

\begin{assumption} \label{assum:simple_top_stable}
From now on, we will assume that all the systems we consider are {\bf simple and topologically stable}. In particular, simplicity implies that for the semi-toric systems we consider, focus-focus leaves will only have one vanishing cycle.
\end{assumption}

\subsubsection{Statement of singular Arnold-Liouville theorem}

Now we can formulate an extension of Liouville-Arnold-Mineur theorem to singular leaves. We call the next theorem a ``leaf-wise'' result, as we obtain a normal form for a \emph{leaf} of the system. While assertion $2.$ of the theorem extends Eliasson normal form, item $3.$ gives te ransom of such generalisation: here the normal form of the leaf doesn't preserve the symplectic structure.

\begin{theorem}[Arnold-Liouville with semi-toric singularities,~\cite{Zung-SymplecticTopology_I-1996}] \label{theo:A-L_with_sing}
Let $F$ be a proper semi-toric system, $\Lambda$ be a leaf of Williamson type $\Bbbk$ and $\cV(\Lambda)$ a saturated neighborhood of $\Lambda$ with respect to $F$.

Then the following statements are true:
\begin{enumerate}
 \item There exists an effective Hamiltonian action of $\TT^{k_\ue +\kf +\kx}$ on $\cV(\Lambda)$. There is a locally free $\TT^{\kx}$-subaction. The number $k_{\ue } +k_{\uf} +\kx$ is the maximal possible for an effective Hamiltonian action.
 \item If $\Lambda$ is topologically stable, $(\cV(\Lambda),\cF)$ is leaf-wise homeomorphic (and even diffeomorphic) to an almost-direct product of elliptic and focus-focus simple singularities:
 $(\cV(\Lambda),\cF) \simeq$
 \[ ( \cU(\TT^{\kx}),\cF_r ) \times \cL^\ue_1 \times \ldots \times \cL^\ue_{k_\ue } \times \cL^f_1 \times \ldots \times \cL^f_{k_{\uf}}  \]
where $( \cU(\TT^{\kx}),\cF_r )$ is a regular foliation by tori of a saturated neighborhood of $\TT^{\kx}$.

 \item There exists partial action-angle coordinates on $\cV(\Lambda)$: there exists a diffeomorphism $\varphi$ such that
 \[ \varphi^* \omega = \sum_{i=1}^{\kx} d\theta_i \wedge dI_i + P^* \omega_1 \]
 where $(\Btheta,\bI)$ are the action-angle coordinates on $T^* \cT$, where $\cT$ is the zero torus in Miranda-Zung equivariant Normal form theorem stated in~\cite{MirandaZung-Equivariantnormalform-2004} and $\omega_1$ is a symplectic form on $\RR^{2(n-\kx)} \simeq \RR^{2(k_{\ue } + 2k_{\uf})}$.
\end{enumerate}
\end{theorem}

Theorem~\ref{theo:A-L_with_sing} says that in particular for semi-toric systems, under this mild assumption that is topological stability on the leaves, a critical leaf $\Lambda$ is diffeomorphic to a product of the ``simplest'' regular, elliptic and focus-focus leaves one can find. 

\begin{remark}
 Assertion $3.$ of Theorem~\ref{theo:A-L_with_sing} explains that {\bf we do not have a symplectomorphism} in the assertion $2.$
 
 Also, one should notice that in Theorem~\ref{theo:A-L_with_sing}, it is only because we made the Assumption~\ref{assum:simple_top_stable} that the singularity is leaf-wise {\bf diffeomorphic} to an almost-direct product of simple sigularities. For instance, were there more than one pinch on the singuarity, one could only guarantee the existence of an homeomorphism between the two.
\end{remark}

This theorem is the last generalisation of Eliasson's point-wise normal form. As of now, we shall speak of Eliasson-Miranda-Zung normal form and precise if necessary what type of normal form we are working on.

\subsection{Stratification by Williamson index}

 The stratification of $M$ by the rank of the moment map is an important feature in the study of toric system in the theorem of Atiyah - Guillemin \& Steinberg. The following results are consequences of Theorem~\ref{theo:A-L_with_sing}, and of its proof in~\cite{Zung-SymplecticTopology_I-1996}. By stratification, we mean here the following

\begin{definition}

A stratified manifold is a triplet $(M,\cS,\cI)$ with $M^n$ topological manifold  is a finite partition $\cS=(S_i)_{i \in \cI}$ of $M^n$ indexed by an ordered set $(\cI,\preceq)$ such that  

\begin{itemize}
 \item Decomposition: for $i \in \cI$ the $S_i$'s are smooth manifolds and for $i,j \in \cI$, $i \preceq j$ if and only if $S_i \subseteq \overline{S_j}$.
 
 \item Splitting condition: by induction over the dimension of the stratified manifold $n$.

 If $x \in S_i$, for a neighborhood $U_x$ of $x$ in $\RR^n$, there exists a disk $D_i \subset \RR^{\dim(S_i)}$ and a cone $\cC(J) = $ over a $(n − \dim{S_i} − 1)$-dimensional stratified smooth manifold $J$ (and $(n-\dim{S_i}-1) \leq n-1$ ) such that:
$U_x$ and $D_i \times \cC(J)$ are isomorphic as stratified manifolds.

 
\end{itemize}

\end{definition}

The definition above is consistent, since strafications of manifolds of dimension $0$ and $1$ are obvious. Usually, one defines stratifications over larger classes of objects that are not regular enough for the given topology, while the strata are. The example to remember is the cone, or the manifolds with corners: they won't be smooth manifold but their strata will be.

\begin{theorem} \label{theo:strat_Williamson_M}
 
 Let $(M^{2n},\omega,F)$ be a semi-toric integrable Hamiltonian system. 
 
 Then we have the following results:
 
 \begin{itemize}
  \item For any $\Bbbk \in \cW(F)$, each connected component of $P_\Bbbk(M)$ is a smooth open symplectic manifold diffeomorphic to $T^* \TT^\kx$. The symplectic form is given by the partial \act coordinates on it.
  
  \item The triplet $(P_\Bbbk (M), \omega_\Bbbk, F_\Bbbk)$, with $F_\Bbbk:= F|_{P_\Bbbk (M)}$ is an integrable Hamiltonian system with no critical points.
 
 \item  The triplet $(M,P_\Bbbk(M),\cW(F))$ is a stratified manifold. 

 \end{itemize}

\end{theorem}

\begin{proof}
 
 Concerning points $1.$ and $2.$, if we take a critical point $p \in M$ of Williamson index $\Bbbk$, with items $2.$ and $3.$ of Theorem~\ref{theo:A-L_with_sing} we have a tubular neighborhood $\cV$ of the leaf containing $p$ such that $\cV$ is leaf-wise diffeomorphic to \[ ( \cU(\TT^{\kx}),\cF_r ) \times \cL^\ue_1 \times \ldots \times \cL^\ue_{k_{\ue }} \times \cL^f_1 \times \ldots \times \cL^f_{k_{\uf}}, \]
 and a diffeomorphism $\varphi$ such that $\varphi^* \omega = \sum_{i=1}^{\kx} d\theta_i \wedge dI_i + P^* \omega_1$. The subset $P_\Bbbk (\cV)$ is diffeomorphic in these local coordinates to $\{ q_1 = q_2 = 0, x^\ue_1 = y^\ue_1 = \ldots = x^\ue_{k_\ue} = y^\ue_{k_\ue} = 0 \}$, that is, to an open subset of $T^* \TT^{\kx}$ of the form $\TT^{\kx} \times \mathring{D}^{\kx}$. On it, the 2-form $\omega_1$ in Zung's theorem vanishes so $P_\Bbbk (\cV)$ is described by the partial action-angle coordinates given by Miranda and Zung in Theorem~\ref{theo:M-Z}. This proves item $1.$.
 
 Next, we have that $F_\Bbbk \in \smooth(P_\Bbbk(M) \to \RR^{\kx})$, it is clearly integrable as an Hamiltonian system of $P_\Bbbk(M)$. Moreover, a critical point for $F_\Bbbk$ is a critical point for $F$ with a smaller Williamson index, which is impossible on $P_\Bbbk (M)$ by definition. Hence $F_\Bbbk$ has no critical point on $P_\Bbbk(M)$. This proves point $2.$.
 
 To conclude with point $3.$, we first have to prove the decomposition condition. it is clear that $\{ P_\Bbbk (M) \}_{\Bbbk \in \cW(F)}$ is a partition of $M$ by smooth manifolds. The indexing set is the poset $\cW(F)$. Remebering Proposition~\ref{prop:closure_orbit}, for $\Bbbk, \Bbbk'$ in $\cW(F)$, if $P_\Bbbk(M) \subseteq \overline{P_{\Bbbk'}(M)}$, then $\Bbbk \preccurlyeq \Bbbk'$. To prove the converse statement, one can notice with the local models that for a critical point $m$ of Williamson index $\Bbbk$, $m$ can always be attained by a sequence of points that have the same Williamson index $\Bbbk'$, provided that $\Bbbk' \succcurlyeq \Bbbk$. In particular, $m$ is in $\overline{P_{\Bbbk'} (M)}$. 
 
 For the splitting condition, with Item $2.$ of Theorem~\ref{theo:A-L_with_sing} we see that we only need to treat the simple elliptic and focus-focus cases with local models. In the elliptic case $P_E{\RR^2}$ is just a point: a neighborhood of the critical point is a disk, it is homeomorphic to the critical point times a cone over a small circle. For the focus-focus case, it is not more complicated: $P_{FF}(\RR^4)$ is again a point, and we need to show there exists a $3$-dimensional stradispace $L$ such that a neighborhood of the focus-focus point is homeomorphic to this point times the cone over $L$. We can just take the $3$-sphere $S^3$ and take the cone over it: it is homeomorphic to the $4$-ball, and hence is a neighborhood of a focus-focus point.
 
\end{proof}

There is a natural way to refine the stratification, by distinguishing the connected components of the $P_\Bbbk(M)$. We can note $P^l_\Bbbk(M)$, with $l = 1..L(\Bbbk)$ an integer labelling the $l$-th connected components of $P_\Bbbk(M)$. The new ordering set is thus obtained from $\cW(F)$ by ``splitting'' each $\Bbbk$: we consider the couples $(\Bbbk,l)$, and write $(\Bbbk,l) \overline{\preccurlyeq} (\Bbbk',l')$ if and only if $\Bbbk \preccurlyeq \Bbbk'$ and $P^l_\Bbbk \subseteq \overline{P^l_\Bbbk(M)}$. The proof of the stratification is exactly the same since the spliting condition is a local condition.

Now, if $\kf = 1$, we can take the skeleton of a strata, $P^l_\Bbbk (M)$. Here, the skeleton is obtained by patching smoothly pieces of dimension $\leqslant \kx-2$, so we can extend $\omega_\Bbbk$ and $F_\Bbbk$ such that $(P^l_{\preccurlyeq \Bbbk},\omega_{\preccurlyeq \Bbbk},F_{\preccurlyeq \Bbbk})$ is an integrable Hamiltonian system on a closed manifold. This system is toric: the singularities are non-degenerate and without hyperbolic component (such component would have to come from the total system) or focus-focus component (because $\kf = 1$, such point would be a critical point with $\kf \geqslant 2$ on $(M,\omega,F)$, which is impossible). Lastly, since the Williamson type is invariant by the Hamiltonian flow, $F_\Bbbk$ yields a $\TT^\kx$-action on $P_{\preccurlyeq \Bbbk}(M)$.

\section{ Local models for semi-toric values } \label{section:loc_mod_ST_val}

Eliasson normal form theorem~\ref{theo:Eliasson_NF-ST_case} and Miranda-Zung theorem~\ref{theo:M-Z} give us a general understanding of the commutant of a semi-toric system: we have a local model of the image of the moment map near a critical value, that can be precised in the semi-toric case. To do so, we remind first these two lemmas. 

Two Lagrangian foliations $\cF$ and $\cF'$ are said to be equivalent if there exists a symplectomorphism such that for each $L \in \cF$, $\varphi(L) = L' \in \cF'$.

\begin{lemma} \label{lemma:exist_G}
Let $\mathcal{F}$ be a singular Liouville foliation given by a momentum map $F:M \to \RR^k$. Let $\mathcal{F}'$ be a singular Liouville foliation given by a momentum map $F':N \to \RR^k$.
If the level sets are locally connected, then for every smooth symplectomorphism $\varphi: \cU \subset M \to \cV \subset N $ where $\cU$ is an open neighborhood of $p \in M$, $\cV$ a neighborhood of $p' = \varphi(p) \in N$, and such that \[ \varphi^* \mathcal{F} = \mathcal{F}' ,\] there exists a unique local diffeomorphism $ G: (\RR^k, F(p)) \to (\RR^k, F'(p')) $ such that 

\[ F \circ \varphi = G \circ F' .\]
\end{lemma}

\begin{lemma} \label{lemma:finite_number_cc_fiber}
 Let $(M,\omega,F)$ be a proper almost-toric system. Then its fibers has a finite number of connected components.
\end{lemma}
\begin{proof}{\emph{of Lemma }} \ref{lemma:finite_number_cc_fiber}
 Let $c$ be a value of $F$, and $L$ be a connected component of $F^{-1}(c)$. On each point of $L$ we can apply Eliasson normal form. The leaf $L$ is compact and there is only a finite number of local models that can occur on $L$ (actually, there can be only only two different local models on a same leaf), this gives us the existence of an open neighborhood $\cV(L)$ of $L$ in which there is no other connected components of $F^{-1}(c)$: the leaves are separated in $F^{-1}(c)$. Another way to look at it is to take the saturated neighborhood of the leaf in Zung's leaf-wise model of singularities.
 
 Now, we have that $\bigcup_{L \subseteq F^{-1}(c)} \cV(L)$ is an open covering of $F^{-1}(c)$, which is compact by the properness of $F$. We can thus extract a finite sub-covering of it. It implies that there is only a finite number of connected components. 
\end{proof}

\subsection{Symplectomorphisms preserving a semi-toric foliation}

Eliasson-Miranda-Zung normal form gave us the existence of $G$ a local diffeomorphism of the linear model associated to any integrable system. The theorem presented here gives precisions about the form of $G$ in the semi-toric case

\begin{theorem} \label{theo:pres_semi-toric_G}
Let $F$ be a semi-toric integrable system, and $m$ be a critical point of Williamson type $\Bbbk$, with $k_{\uf} = 1$. Let $\cU_m$ be a \neigh of $m$, $\varphi: (\cU_m, \omega) \to ( \varphi(\cU_m) \subseteq L_\Bbbk, \omega_\Bbbk)$ a symplectomorphism sending the transverse orbit of $m$ on the zero-torus and such that the foliations $\varphi^* \cF$ and $\cQ_\Bbbk$ are equivalent on $\cU_m$.

Then the diffeomorphism $G: U \to G(U) \subseteq F(M)$ given by Eliasson-Miranda-Zung normal form and lemma~\ref{lemma:exist_G} is such that:

\begin{enumerate}
 \item $ F \circ \varphi^{-1} = G \circ Q_\Bbbk$, 
 \item $ \check{F} = A\cdot  \check{Q}_\Bbbk + \check{F}(c)$, with $A \in GL_{n-1} (\ZZ)$.
\end{enumerate}
  
\end{theorem}

In the future, we shall dinstinguish different parts of $A$, so we shall write the Jacobian matrix of $G$ as:

\[ \begin{pmatrix}
   \partial_{q_1} \! G_1 & \partial_{q_2} \! G_1 & \partial_{q^{(1)}_e} \! G_1 & \ldots & \!\! \partial_{q^{(k_{\ue })}_\ue} \! G_1 & \partial_{I_1} \! G_1 & \ldots & \partial_{I_{\kx}} \! G_1 \\
   0 & F^\uf  & F^\ue_1 & \ldots & F^\ue_{k_{\ue }} & F^\ux_1 & \ldots & F^\ux_{\kx} \\
   0 & E^\uf_1 & \ulcorner & \; & \urcorner & \ulcorner & \; & \urcorner \\
   \vdots & \vdots & \; & E^\ue & \; & \; & E^\ux & \; \\
   \vdots & E^\uf_{k_{\ue }} & \llcorner & \; & \lrcorner & \llcorner & \; & \lrcorner \\
   \vdots & X^{\uf}_1 & \ulcorner & \; & \urcorner & \ulcorner & \; & \urcorner \\
   \vdots & \vdots & \; & X^\ue & \; & \; & X^\ux & \; \\
   0 & X^{\uf}_{\kx} & \llcorner & \; & \lrcorner & \llcorner & \; & \lrcorner \\
   \end{pmatrix}   ( \spadesuit ).
\]

Put in another way, \[ A = \begin{pmatrix} F^\uf & F^\ue & F^\ux \\ E^\uf & E^\ue & E^\ux \\ X^{\uf} & X^\ue & X^{\ux} \end{pmatrix} ,\] with:
\begin{itemize}
 \item[$\bullet$]  $F^\uf \in\ZZ$, $F^\ue \in \ZZ^{1 \times \ke}$, $F^\ux \in \ZZ^{1 \times \kx} $,
 \item[$\bullet$]  $E^\uf \in \ZZ^{\ke \times 1} $, $E^\ue \in \ZZ^{\ke \times \ke}$, $E^\ux \in \ZZ^{\ke \times \kx} $,
 \item[$\bullet$]  $X^{\uf} \in M_{\kx \times 1} (\ZZ) $, $X^\ue \in M_{k_{\ue }} (\ZZ) $, $X^{\ux} \in M_{\kx} (\ZZ) $.
\end{itemize}

Note that this theorem includes the particular case where $F= Q_\Bbbk$, that is, when we consider symplectomorphisms that start and end with $Q_\Bbbk$.


\begin{proof}[Proof of Theorem]\ref{theo:pres_semi-toric_G}

Since $\varphi$ is a symplectomorphism, it preserves the dynamics induced by the Hamiltonian vector fields of $\check{F}$. The open set is alwas connected. We have assumed that the components of $\check{F}$ have $2\pi$-periodic flows. So, once pushed forward by $\varphi$, the vector fields must remain $2\pi$-periodic. We have the expression

\[ \chi_{f_i \circ \varphi^{-1}} = \partial_{q_1} G_i .\chi_{q_1} + \partial_{q_2} G_i .\chi_{q_2} + \sum_{j=1}^{k_{\ue }} \partial_{q_{(j)}^\ue} G_i .\chi_{q_{(j)}^\ue} + \sum_{j=1}^{\kx} \partial_{I_j} G_i .\chi_{I_j}
\]

The partial derivatives of $G$ are constant under the action by the Hamiltonian flow of $Q_\Bbbk$. The classical complex variables of singular flows are the following 

\[
\begin{array}{||c||c||c||}
 \hline
 \text{Type} & \text{Elliptic} &\text{Focus-Focus} \\
 \hline
 \text{Coordinates} & z^\ue:= x^\ue + i \xi^\ue & \begin{cases}
                      z^{\uf_1}:= x^\uf_1 + i x^\uf_2 \; , \; z^{\uf_2}:= \xi^\uf_1 + i \xi^\uf_2 \\
                      q^\uf:= \bar{z}^{\uf_1} z^{\uf_2}:= q^{\uf_1} + i q^{\uf_2}
                      \end{cases} \\ 
 \hline
 \text{Formula} & \phi_{q^\ue}^t (z^\ue) = e^{iq^\ue t} z^\ue 
 
 &                    \begin{cases}
                      \phi^t_{\uf_1} (z_1,z_2):= (e^{- q^{\uf_1} t} z_1, e^{ q^{\uf_2} t} z_2) \\
                      \phi^t_{\uf_2} (z_1,z_2):= (e^{ i q^{\uf_2} t} z_1, e^{ i q^{\uf_2} t} z_2)
                      \end{cases}
\\
\hline
\end{array}
\]

From this we deduce the expression of the flow on a neighborhood $\cU$:
\[
\begin{aligned}
& \phi^t_{f_i \circ \varphi^{-1} } \left(z_1,z_2,{\bold z^\ue},\Btheta,\bI \right) = \\ & \text{\Big{(}} e^{( \partial_{q_1} G_i + i \partial_{q_2} G_i)t } z_1, e^{(- \partial_{q_1} G_i + i \partial_{q_2} G_i)t } z_2, e^{i \partial_{{\bold q}^\ue} G_i .t} \bullet {\bold z}^\ue, \Btheta + \partial_{\bI}G_i.t,,\bI \text{\Big{)}}. \\
\end{aligned}
\]
So necessarily, for $i =2,\ldots,n$ and $j =1,\ldots,n$, we have $\partial_{q_1} G_i = 0$ and 

\[ \partial_{q_2}G_i , \partial_{q_{j}^\ue}G_i \in \ZZ \text{ and } \partial_{I_{\kx}}G_i \in \ZZ .\]

If a coefficient of the Jacobian is integer on $\varphi (\cU)$, it must be constant on it. This shows that $ \check{F} = A \circ \check{Q}$ with $A \in M_{n-1} (\ZZ)$. Now, $A$ is invertible because $G$ is a local diffeomorphism, and since the components of $\check{F}$ are all $2\pi$-periodic, we have that necessarily $A^{-1} \in M_{n-1} (\ZZ)$.

\end{proof}

Note that this result here only uses the $2\pi$-periodicity of the flow, and no other assumption about the dynamics of $F$. In the next theorem, we have the \emph{same} foliation before and after composing with $\varphi$. This stronger statement will get us precisions about the form of $Jac(G)$, in particular the unicity of its infinite jet on the set of critical values.

\subsection{Transition functions between the semi-toric local models}

In this section, we need to precise our notion of flat function. To this end, let's introduce the following set:

\begin{definition}
 Let $S = \{ (p,v) | v \in T_p M \}$ a subset of the tangent bundle $TM \xrightarrow{\pi} M$ and $\cU$ an open subset of $M$. We define the set $\flat_S (\cU)$ as the set of real-valued smooth functions on $\cU$, such that for each point $p \in \pi(S)$ its infinite jet in the direction $v$ is null (i.e. the jet at any order is null).
\end{definition}

This definition is voluntary general, as we shall consider smooth function that are possibly flat in some directions at one point, in different direction at other points etc. 

\subsubsection{ Symplectomorphisms preserving a linear semi-toric foliation }

We prove here a kind of \emph{uniqueness} theorem of the $G$ introduced in Theorem~\ref{theo:pres_semi-toric_G}. Here, we write the diffeomorphism $B$, as the constraints of $G$ in that case is an information about the possible changes of the \emph{Basis} in the ``variables'' $Q_\Bbbk$ that can occur.

\begin{theorem} \label{theo:pres_trans_B}
 Let $(L_\Bbbk,\omega_\Bbbk,Q_\Bbbk)$ be a singularity in Zung's theory with $k_{\uf} =1$. Let $\psi$ be a symplectomorphism of $\cU \subset L_\Bbbk$ which preserves the foliation $\cQ_\Bbbk$.

 Then the diffeomorphism $B: U \to B(U)$ introduced in~\ref{theo:pres_semi-toric_G} is such that there exists $\epsilon^{\uf}_1,\epsilon^{\uf}_2 \in \{ -1,+1 \}$, a matrix $\Beps^\ue \in \Diag_{\ke} (\{-1,+1\}) $ and a function $u \in \flat_S (U)$, where $S:= TM|_{V_\Bbbk (\cU)}$, so that we have:
 
 \begin{enumerate}
  \item $ Jac(B)_1 (Q_\Bbbk) = (\epsilon^{\uf}_1 q_1 + \partial_{q_1}u, \partial_{q_2}u, 0,\ldots,0)$, 
  \item $E^\uf = 0$, $E^\ue = \Beps^\ue$ and $E^\ux = 0 $,
  \item $F^\uf = \epsilon^{\uf}_2$, $F^\ue =0$ and $F^\ux = 0$.
  \item $X^{\uf} \in M_{\kx,1} (\ZZ) $, $X^\ue \in M_{k_{\ue }} (\ZZ) $, $X^{\ux} \in GL_{\kx} (\ZZ) $,
 \end{enumerate}
 
 That is, for $\tilde{x} = x \circ \psi^{-1}$ : 

\begin{enumerate}
 \item[a.] $\tilde{q_1} = \epsilon^{\uf}_1 q_1 + u$ , $\tilde{q_2} = \epsilon^{\uf}_2 q_2$
 \item[b.] ${\tilde{q}}^{(i)}_e= \epsilon^\ue_i q^{(i)}_e$
 \item[c.] $(\tilde{I}_1,\ldots,\tilde{I}_{\kx}) = \check{Q_\Bbbk}$ 
\end{enumerate}

For the Jacobian, it means that  

\[ Jac(B) = 
\begin{pmatrix}
   \epsilon^{\uf}_1 + \partial_{q_1} u & \partial_{q_2} u & \ldots & \ldots & \ldots & \ldots & \ldots & \partial_{e_n} u \\
    0 & \epsilon^{\uf}_2 & 0 & \ldots & 0 & 0 & \ldots & 0 \\

    0 & 0 & \epsilon^\ue_1 & \ddots & \vdots & 0 & \ldots & 0 \\
   \vdots & \vdots & \ddots & \ddots & 0 & \vdots & \ldots & \vdots \\
   0 & 0 & \ldots & 0 & \epsilon^\ue_{k_{\ue }} & 0 & \ldots & 0 \\

   \vdots & X^\uf_1 & \ulcorner & \; & \urcorner & \ulcorner & \; & \urcorner \\
   \vdots & \vdots & \; & X^\ue & \; & \; & X^\ux & \; \\
   0 & X^\uf_{\kx} & \llcorner & \; & \lrcorner & \llcorner & \; & \lrcorner \\
   \end{pmatrix} ( \bigstar )
\]
with the $u_i$'s being in $\flat_S ( \cU )$ (the $u_i$'s are the derivates of $u$ in all the variables).

\end{theorem}

For practical uses, we are not interested in the precise form of $u$, but just by the fact that it is flat on $S$ as defined above.

The theorem was proved in the focus-focus case for $2n=4$ by \SVN in~\cite{SVN-semiglobalinvariants-2003}. We follow the same ideas to give here a proof in the general case.

\begin{proof}[Proof of Theorem]\ref{theo:pres_trans_B}

As a particular case of theorem~\ref{theo:pres_semi-toric_G}, we already know that $Jac(B)$ is of the form $(\spadesuit)$. The point here is to exploit the fact that the linear model $Q_\Bbbk$ has specific dynamical features conserved by a canonical transformation. Note also that the fibers of $Q_\Bbbk$ are connected.

{ \bf Points (2), (3) and (4):}

A point fixed by the flow of a Hamiltonian $H$ is preserved by a symplectomorphism: its image will be a fixed point for the precomposition of $H$ by a symplectomorphism. Theorem~\ref{theo:M-Z} tells us how critical loci of a given Williamson type come as ``intersections'' of other critical loci. In particular, we know that for each $i=1,\ldots,k_{\ue } $, there exists a $c_i \in U$ such that for all $ p \in (Q_\Bbbk)^{-1} (c_i) \cap \cU$, $ \chi_{q^\ue_i} (p) = 0$, $\chi_{q^\ue_{j}} (P) \neq 0$ for $j\neq i$ and $\chi_{q_2} (p) \neq 0$ ( the $\chi_I$'s never cancel). Now, we have the following formula
\[ \chi_{q^\ue_{i} \circ \psi^{-1}}(p) = E^\uf_i \chi_{q_2}(p) + \sum_{j=1 , j \neq i}^{k_\ue } E^\ue_{ij} \chi_{q^\ue_j} (p) + \sum_{j=1}^{k_{\uf}} E^\ux_{ij} \chi_{I_j} (p)= 0 .\]

The family

\[ (\chi_{q_2} (p),\chi_{q^{(1)}_e} (p),\ldots,\chi_{q^{(i-1)}_e} (p),\chi_{q^{(i+1)}_e} (p),\ldots,\chi_{q^{(k_{\ue })}_e} (p),\chi_{I_1},\ldots,\chi_{I_{\kx}}) \]
is a free family ($Q_\Bbbk$ is an integrable system), so we have for $j \neq i$ that $E^\ue_{ij} = 0 $, and $E^x_{ij} = 0$ for all $j=1,\ldots,\kx$. This is true for all $i=1,\ldots,k_{\ue } $.

If we take a $c$ in $\{ q_1 = q_2 = 0 \}$, we get by a similar reasoning that $F^\ue = 0$ and $F^\ux = 0$. 


\ouf

{\bf Point (1), part 1: }

Now, we need to deal with the flow of $q_1$, that is, the ``pseudo-hyperbolic'' component of the focus-focus singularity.

On an open set $\cU$, we have the explicit expression for the field of $q_1 \circ \psi^{-1}$

\[ \chi_{q_1 \circ \psi^{-1}} = \partial_{q_1} B_1 \cdot  \chi_{q_1} + \partial_{q_2} B_1 \cdot  \chi_{q_2} + \sum_{j=1}^{k_{\ue }} \partial_{q^{(j)}_e} B_1 \cdot  \chi_{q^{(j)}_e} + \sum_{j=1}^{\kx} \partial_{I_j} B_1 \cdot  \chi_{I_j}
 .\]

When evaluated on $p \in \overline{P^{Q_\Bbbk}_\Bbbk (\cU) } $, comes 

\[0 = \partial_{q_2} B_1 \cdot  \chi_{q_2} (p) + \sum_{j=1}^{k_{\ue }} \partial_{q^\ue_j} B_1 \cdot  \chi_{q^\ue_j} (p) + \sum_{j=1}^{\kx} \partial_{I_j} B_1 \cdot  \chi_{I_j} (p)
 .\]

First, since the $\chi_{I_i}$'s have no fixed point, we necessarily have that 

\begin{equation} \label{equ:X_component_q1_zero_CrV_Bbbk}
 \forall c \in V^{Q_\Bbbk}_\Bbbk (\cU) \text{ and } \forall j=1,\ldots,\kx \ ,\  \partial_{I_j} B_1 (c) = 0. 
\end{equation}

The result is true for all $\Bbbk$ with $\kf =1$. If $k_\ue \geqslant 1$, we can apply the same reasoning to $V_{\Bbbk'} (\cU)$, where $\Bbbk'=(0,k_{\uf},0,\kx+k_{\ue}) \succcurlyeq \Bbbk$ and get equation~\ref{equ:X_component_q1_zero_CrV_Bbbk} for this set. Since $V^{Q_\Bbbk}_\Bbbk (U) \subseteq \overline{V^{Q_\Bbbk}_{\Bbbk} (U)}$, we have

\begin{equation} \label{equ:E_component_q1_zero_CrV_Bbbk}
 \forall c \in V^{Q_\Bbbk}_{\Bbbk} (U) \text{ and } \forall j=1,\ldots,k_{\ue } \ ,\  \partial_{q^\ue_j} B_1 (c) = 0 .
\end{equation}

Now that we know there is no transverse nor elliptic component in the flow of $q_1 \circ \psi^{-1}$ for critical leaves with focus-focus component, let's focus on the $q_2$-component. 

A leaf $\Lambda'$ of $\cQ_\Bbbk$ of Wiliamson type $\Bbbk' \preccurlyeq \Bbbk$ is stable by the flow of $q_1$. On it, the flow is radial: for a point $m' \in \Lambda$ of Williamson type $\Bbbk' =(0,1,\kx)$, there exists a unique point $m$ on the zero-torus of the leaf $\Lambda$ such that the segment $[m',m[$ is a trajectory for $q_1$. Depending on whether $m'$ is on the stable (+) or the unstable (-) manifold, we have that $[m',m[ = \{ \phi^{\pm t}_{q_1}(m') \ | \ t \in [0,\infty[ \}$. Remembering that $m$ is a fixed point for $q_1$, we have that $\psi([m',m[)$ is be a trajectory of $q_1 \circ \psi^{-1} =  B_1 \circ Q_\Bbbk $.

The image trajectory $\psi([m',m[)$ is contained in a 2-dimensional plane (the stable or unstable manifold). Since $\psi$ is smooth at $m'$, $\psi([m',m[)$ is even contained in a sector of this plane that also contains $m'$. Remembering that $\psi([m',m[)$ is a trajectory \emph{ for an infinite time}, the only linear combinations of $\chi_{q_1}$, $\chi_{q_2}$ which yields trajectories confined in a fixed sector are multiples of $\chi_{q_1}$ only. So we have that

\begin{equation} \label{equ:q2_component_q1_zero_CrV_Bbbk}
  \partial_{q_2} B_1 (c) = 0 .
\end{equation}

This shows that $B_1$ is constant in the variables $(\Bq^\ue,\bI)$ on $V_\Bbbk (U)$, but does not tell us more information.

{ \bf Point (1), part 2: }

To show that $B_1-q_1$ is flat on $V_\Bbbk (U)$ in the variables $q^\uf_1$ and $q^\uf_2$, we can now treat the variables $(\Bq^\ue,\bI)$ as parameters. We can always suppose that $\psi$ preserves the stable and unstable manifolds of $q_1$: this amounts to fix the sign of $\partial_1 B_1$ to be positive on $\cU$. As a result we'll have $\epsilon^{\uf}_1 =1$. And again, we can assume that $k_{\ue } =0$, as flatness is a closed property: here, it is stable when taking the limit $\bold{q}^\ue \to 0$.

With the explicit expression of the flow of $q_1$ and $q_2$, if we set $\bar{z_1} z_2 = c$ and $z_2 = \bar{\delta}$, we have for the joint flow of $q_1$ and $q_2$ at respective times $s = \ln \abs{\frac{\delta}{c}}$ and $t= \arg(\delta) - \arg(c)$

\begin{equation} \label{equ:jointflow}
\underbrace{\phi_{q_1}^s \circ \phi_{q_2}^t}_{ =: \Upsilon} (c,\bar{\delta},\Btheta,\bI) = (\delta,\bar{c},\Btheta,\bI)
\end{equation}

One can then state the fact that $\Upsilon$ is a smooth and single-valued function in a neighborhood $W$ containing $\{ (0,\bar{\delta},\Btheta,\bI) , \Btheta \in \TT^{\kx} , \bI \in \cB^{\kx} (0,\eta) \}$. Now, we know that $\psi^{-1} (0,\bar{\delta},\Btheta,\bI) $ is of the form $(0,a,\Btheta',\bI)$ and $\psi^{-1} (\delta,0,\Btheta,\bI) $ is of the form $(b,0,\Btheta'',\bI)$, since $\psi$ preserve the level sets and the stable and unstable manifolds. Hence, for $\psi^{-1} \circ \Upsilon \circ \psi$

\[ (0,a,\Btheta',\bI)  \overset{\psi}{\mapsto} (0,\bar{\delta},\Btheta,\bI)  \overset{\Upsilon}{\mapsto} (\delta,0,\Btheta,\bI) \overset{\psi^{-1}}{\mapsto} (b,0,\Btheta'',\bI) .\]

With the expression of $\Upsilon$ in~\eqref{equ:jointflow}, we know that in the complementary set of $\{z_1 = 0\}$, $\psi^{-1} \circ \Upsilon \circ \psi$ is equal to the joint flow of $B(q_1,q_2,\bI)$ at the multi-time $(ln \abs{\frac{\delta}{c}},arg(\delta) - arg(c),0,\ldots,0)$. With what we already know about the flow of $q_2 \circ \psi^{-1}$, when we write the joint flow in terms of the flows of the components of $Q_\Bbbk$ we get:

\[
\begin{aligned}
 \phi_{B(q_1,q_2,\bI)}^{(ln \abs{\frac{\delta}{c}},arg(\delta) - arg(c))} & = \phi_{\partial_1  B_1 \cdot q_1 + \partial_2  B_1 \cdot q_2 + \sum_j \partial_{I_j}  B_1 \cdot  I_j}^{ ln \abs{\frac{\delta}{c}} }  \circ \phi_{q_2}^{arg(\delta) - arg(c)} \\
 & = \phi_{q_1}^{\partial_1  B_1 \cdot  ln \abs{\frac{\delta}{c} }} \circ \phi_{q_2}^{ \partial_2  B_1 \cdot  ln \abs{\frac{\delta}{c}} + arg(\delta) - arg(c) } \\ 
 \; &  \hspace{35mm} \circ \underbrace{\phi_{I_1}^{ - \partial_{I_1} B_1 \cdot  ln \abs{ \frac{c}{\delta}} } \circ \ldots \circ \phi_{I_{\kx}}^{ - \partial_{I_{\kx}} B_1 \cdot  ln \abs{ \frac{c}{\delta}} } }_{ = id \text{ (c.f. Point 4.(1))}}
\end{aligned}
.\]

Since $\psi^{-1} \circ \Upsilon \circ \psi$ is smooth at the origin, it's also smooth in a neighborhood of the origin; for $c$ small enough, we can look at the first component of the flow on $(c,a,\Btheta,\bI)$: here $c$ shall be the variable while $a$, $\Btheta$ and $\bI$ are parameters. We have the application

\[
\begin{aligned}
c \mapsto & e^{\partial_1  B_1 \cdot  ln \abs{\frac{\delta}{c}} + i ( \partial_2  B_1 . ln \abs{\frac{\delta}{c}} + arg(\delta) - arg(c) ) } c \\
 = & \left[ e^{\partial_1  B_1 ln|\delta| + i ( \partial_2  B_1 ln|\delta| + arg(\delta) )} \right] e^{(1-\partial_1  B_1) \cdot ln|c| + i ( - \partial_2  B_1 ln |c| - arg(c) )} .
\end{aligned}
\]

The terms in brackets are obviously a smooth function of $c$, and so the last exponential term is also smooth as a function of $c$ on 0. Hence, the real part and the imaginary part are both smooth functions of $(c_1,c_2)$ in $(0,0,\bI)$. 

We then have the following lemma:

\begin{lemma}
 Let $f \in \mathcal{C}^{\infty}(\RR^k \to \RR)$ be a smooth function such that: $x \mapsto f(x)~ ln\parallel\!x\!\parallel$ is also a smooth function.

 Then $f$ is necessarily flat in 0 in the $k$ variables $(x_1,\ldots,x_k)$.
\end{lemma}

With this elementary lemma, we have that $(1 - \partial_1  B_1) \circ \alpha $ and $\partial_2  B_1 \circ \alpha$ are flat for all $(0,0,\bI)$, where $\alpha(c_1,c_2,\bI) = (c_1 \delta_1 + c_2 \delta_2, c_1 \delta_2 - c_2 \delta_1 , \bI)$. The function $\alpha$ is a linear function, and it is invertible since $\delta \neq 0$. This gives us the flatness of $1- \partial_1  B_1$ and $\partial_2  B_1$, as functions of $c_1$ and $c_2$ for all the $(0,0,\bI)$, and thus, as functions of all the $n$ variables for all the $(0,0,\bI)$. We have thus $u \in \flat_S (\cU)$ with $S = \{ (0,0,\bI) | \bI \in \RR^{\kx} \}$. 


Supposing that $\psi$ exchanges stable and unstable manifolds yields the same demonstration \emph{mutatis mutandis}, that is, in the last part of the proof, we look at the first component of the flow on $(\bar{c},a,\Btheta,\Bxi)$ with $a$, $\Btheta$ and $\Bxi$ understood as parameters.
\end{proof}

In Theorem~\ref{theo:pres_semi-toric_G}, we associate to a symplectomorphism that preserves a semi-toric foliation a unique $G$ of the form $( \spadesuit )$. It would be interesting to have more knowledge about the restrictions on such symplectomorphism, but here we want to describe the moment map. We'd like to know for instance to what extent $G$ is unique in Theorem~\ref{theo:pres_semi-toric_G}.

Theorem~\ref{theo:pres_trans_B} can be applied to answer this question. If we think of the $G$'s as ``local models'' or ``charts'' of the image of the moment map, then Theorem~\ref{theo:pres_trans_B} describes the ``transition functions'' $B$. Indeed, if we have two local models given by Theorem~\ref{theo:pres_semi-toric_G}, on $\cU \subseteq M$, a neighborhood of a point $m$ of Williamson type $\Bbbk$, one has:

\[ \begin{cases}
   F \circ \varphi = G \circ Q_\Bbbk, \\
   F \circ {\varphi'} = G' \circ Q_\Bbbk .
   \end{cases}
\]

Then we get: 

\[ Q_\Bbbk \circ \underbrace{( \varphi^{-1} \circ \varphi' )}_{= \psi} = \underbrace{( G^{-1} \circ G' )}_{ = B } \circ Q_\Bbbk .\]

We can apply Theorem~\ref{theo:pres_trans_B} to the pair $( \psi =  \varphi^{-1} \circ \varphi' , B = G^{-1} \circ G' )$ and then get the relation:
  
\begin{equation} \label{equ:rel_unicity_semi-toric}
G' = \left(\epsilon^{\uf}_1 G_1 + u , \epsilon^{\uf}_2 G_2 ,\epsilon^\ue_1 G_3,\ldots,\epsilon^\ue_{k_{\ue }} G_{k_{\ue }+2}, ( X^{\uf} | X^\ue | X^{\ux} ) \circ \check{G}  \right) , \  u \in \flat_S ( U ).
\end{equation}

\subsubsection{Symplectic invariants for the transition functions}

We have given some restrictions on the transition functions $B$ between two local models of a semi-toric system. If we now authorize ourselves to change the symplectomorphism in the local models, we can have a ``nicer'' $G$. This amounts to determine what in $B$ is a semi-local symplectic invariant of the system. Let us set 

\[ E_1 = \frac{1}{2} \begin{pmatrix} 1 + \epsilon^{\uf}_1 & 1 - \epsilon^{\uf}_1 & 0 & 0 \\ -1 + \epsilon^{\uf}_1 & 1 + \epsilon^{\uf}_1 & 0 & 0 \\ 0 & 0 & 1 + \epsilon^{\uf}_1 & 1 - \epsilon^{\uf}_1 \\ 0 & 0 & -1 + \epsilon^{\uf}_1 & 1 + \epsilon^{\uf}_1 \end{pmatrix} 
\]
\[
E_2 = \frac{1}{2} \begin{pmatrix} 1 + \epsilon^{\uf}_2 & 0 & 1 - \epsilon^{\uf}_2 & 0 \\ 0 & 1 + \epsilon^{\uf}_2 & 0 & 1 - \epsilon^{\uf}_2 \\ 1 - \epsilon^{\uf}_2 & 0 & 1 + \epsilon^{\uf}_2 & 0 \\ 0 & 1 - \epsilon^{\uf}_2 & 0 & 1 + \epsilon^{\uf}_2 \end{pmatrix} . \]

\begin{theorem} \label{theo:sympl_inv_semi-toric}
 
Let $\psi$ be a symplectomorphism of $L_\Bbbk$ preserving $Q_\Bbbk$ and $B$ one of the possible associated diffeomorphisms of $\RR^n$ introduced in Theorem~\ref{theo:pres_trans_B}, of the form~$( \bigstar )$. Consider the diffeomorphisms 
\[ \begin{aligned}
   \zeta_B (z_1,z_2,\Bx^\ue,\Bxi^\ue,\Btheta,\bI) & = \\ 
   ( e^{ -i \langle \Btheta , X^{\uf} \rangle } z_1,e^{ - i \langle \Btheta , X^{\uf} \rangle} z_2 & ,e^{ -i \Btheta \cdot (X^\ue)^t} \bullet \Bz^\ue, \Btheta, \bI + Q^\ue _\Bbbk \cdot (X^\ue)^t + q_2 \cdot (X^{\uf})^t )  \;
   \end{aligned}
\]
and 
\[ \eta_B (x_1,\xi_1,x_2,\xi_2,\Bx^\ue,\Bxi^\ue,\Btheta,\bI) = ((x_1,\xi_1,x_2,\xi_2) E^t_B,\Bz^\ue,\Btheta \cdot (X^{\ux})^{-1},\bI \cdot (X^{\ux})^t) \]
where $ E_B = E_1 E_2$.

Then we have that $\zeta_B$ and $\eta_B$ are symplectomorphisms of $L_\Bbbk$ which preserve the foliation $\cQ_\Bbbk$ and

\begin{equation} \label{equ:sympl_inv_semi-toric}
 B \circ Q_\Bbbk  = ( \psi \circ \zeta_B \circ \eta_B )^* \left( ( 1,1, \Beps^\ue_1 ,\ldots, \Beps^\ue_{k_\ue},1,\ldots,1 ) \bullet Q_\Bbbk  + (u \circ Q_\Bbbk ,0,\ldots,0) \right)
\end{equation}
with $u$ as introduced in Theorem~\ref{theo:pres_trans_B}.
\end{theorem}

The symplectomorphisms $\zeta_B$ and $\eta_B$ are admissible modifications of semi-toric local models. If we have two local models $(\varphi,G)$ and $(\varphi',G')$, Theorem~\ref{theo:sympl_inv_semi-toric} tells us we can always modify the symplectomorphism of one of them, for instance $\varphi'$, to get another local model $(\tilde{\varphi}',\tilde{G'})$ such that: 

\[ ( \tilde{G'}^{-1} \circ G ) \circ Q_\Bbbk = \left( ( 1,1, \Beps^\ue_1 ,\ldots, \Beps^\ue_{k_\ue},1,\ldots,1 ) \bullet Q_\Bbbk  + (u \circ Q_\Bbbk ,0,\ldots,0) \right) .\]

\begin{proof}{of Theorem}~\ref{theo:sympl_inv_semi-toric}

First let's prove equation~\ref{equ:sympl_inv_semi-toric}. 

\[ \zeta_B^* (q_1 + i q_2) = \zeta_B^* ( \bar{z_1} z_2 ) = e^{i \Btheta \cdot X^{\uf}} \bar{z_1} e^{-i \Btheta \cdot X^{\uf}} z_2 = \bar{z_1} z_2 = q_1 + i q_2 \]

so $q_1, q_2$ are preserved.

\[ \zeta_B^* Q^{\ue }_\Bbbk = ( |e^{-i\Btheta \cdot (X^{\ue }_{1,.})^t} \cdot z^{\ue }_1|^2,\ldots,|e^{-i\Btheta \cdot (X^{\ue }_{k_{\ue },.})^t} \cdot z^{\ue }_{k_{\ue }}|^2 ) = ( |z^{\ue }_1|^2,\ldots,|z^{\ue }_{k_{\ue }}|^2 ) = Q^{\ue }_\Bbbk \]

\[ \zeta_B^* \bI = q_2 \cdot X^{\uf} + Q^{\ue }_\Bbbk \cdot (X^{\ue })^t + \bI . \]

For $\eta$, we have: $\eta^* Q_\Bbbk = (E_1E_2)^* Q^{\uf}_\Bbbk + (X^{\ux})^* Q^{\ux}_\Bbbk$ , so we can treat each action separately. We can also treat $E_1$ and $E_2$ separately, as the two matrices commutes, and treat only the case when  $\epsilon^{\uf}_1$ (respectively $ \epsilon^{\uf}_2$) is equal to $-1$, for when $\epsilon^{\uf}_i = +1$, $E_i =id$.

\begin{itemize}
\item[$\bullet$]  $\epsilon^{\uf}_1= -1$:

\[ E_1 = \begin{pmatrix}0 & 1 & 0 & 0  \\ -1 & 0 & 0 & 0 \\ 0 & 0 & 0 & 1 \\ 0 & 0 & -1 & 0 \end{pmatrix} \text{ so } E_1 \cdot \begin{pmatrix} x_1 \\ \xi_1 \\ x_2 \\ \xi_2 \end{pmatrix} = \begin{pmatrix} \xi_1 \\ -x_1 \\ \xi_2 \\ -x_2 \end{pmatrix} = \begin{pmatrix} \hat{x}_1 \\ \hat{\xi}_1 \\ \hat{x}_2 \\ \hat{\xi}_2 \\ \end{pmatrix} \]
\[ E_1^* q_1 = \hat{x}_1 \hat{\xi}_1 + \hat{x}_2 \hat{\xi}_2 = - \xi_1 x_1 - \xi_2 x_2 = - q_1= \epsilon^{\uf}_1 q_1 ,\]
\[ E_1^* q_2 = \hat{x}_1 \hat{\xi}_2 - \hat{x}_2 \hat{\xi}_1 = \xi_1 (-x_2) - \xi_2 (-x_1) = q_2. \]

\item[$\bullet$]  $\epsilon^{\uf}_2= -1$:

\[ E_2 = \begin{pmatrix} 0 & 0 & 1 & 0 \\ 0 & 0 & 0 & 1 \\ 1 & 0 & 0 & 0 \\ 0 & 1 & 0 & 0 \end{pmatrix}  \]

\[ E_2^* q_1 = \tilde{x}_1 \tilde{\xi}_1 + \tilde{x}_2 \tilde{\xi}_2 = x_2 \xi_2 + x_1 \xi_1 = q_1 \]
\[ E_2^* q_2 = \tilde{x}_1 \tilde{\xi}_2 - \tilde{x}_2 \tilde{\xi}_1 = x_2 \xi_1 - x_1 \xi_2 = - q_2 = \epsilon^{\uf}_2 q_2 . \]

\end{itemize}

And $(X^{\ux})^* Q^\ux_\Bbbk = \bI \cdot(X^{\ux})^t$.

What is left is to prove the preservation of $\omega$. For $\zeta$ we have:

\[ \zeta_{X^{\uf},X^\ue}^* (\omega) = \zeta^* \omega^\uf_\Bbbk + \zeta^* \omega^\ue_\Bbbk + \zeta^* \omega^\ux_\Bbbk. \]

\[
\begin{aligned}
\zeta^* \omega^\uf_\Bbbk & = \re \left[ ( e^{-i \Btheta \cdot X^{\uf}} dz_1 - i z_1 d\Btheta \cdot X^{\uf} e^{-i \Btheta \cdot X^{\uf}} ) \wedge ( e^{i \Btheta \cdot X^{\uf}} d\bar{z_2} + i \bar{z_2} d\Btheta \cdot X^{\uf} e^{i \Btheta \cdot X^{\uf}} ) \right] \\
\ &  = \re \left[ dz_1 \wedge d\bar{z_2} - d\Btheta \cdot X^{\uf} \wedge i( z_1 d\bar{z_2} + z_2 d\bar{z_1} ) \right] \\
\ &  = \omega^\uf_\Bbbk - d\left[ \ \Btheta \cdot X^{\uf} \ \right] \wedge dq_2 
\end{aligned}
\]

\[ \begin{aligned}     
\zeta^* \omega^\ue_\Bbbk & = \sum_{j=1}^{k_{\ue }}  \im \left[ ( e^{-i\Btheta \cdot (X^\ue_{j,.})^t}  dz^\ue_j - i e^{-i\Btheta \cdot (X^\ue_{j,.})^t} z^\ue_j d\Btheta \cdot (X^\ue_{j,.})^t ) \right. \\
\ & \hspace{40mm} \left. \wedge ( e^{i\Btheta \cdot (X^\ue_{j,.})^t} d\bar{z}^\ue_j + i e^{i\Btheta \cdot (X^\ue_{j,.})^t} \bar{z}^\ue_j d\Btheta \cdot (X^\ue_{j,.})^t ) \right] \\
\ & = \sum_{j=1}^{k_{\ue }}  \im \left[ dz^\ue_j \wedge d\bar{z}^\ue_j - d\Btheta \cdot (X^\ue_{j,.})^t \wedge i( z^\ue_j d\bar{z^\ue_j} + \bar{z^\ue_j} dz^\ue_j ) \right] \\
\ & = \omega^\ue_{\Bbbk} - \sum_{j=1}^{k_{\ue }} d\Btheta \cdot (X^\ue_{j,.})^t \wedge dq^\ue_j
\end{aligned}
\]

\[ \begin{aligned}     
\zeta^* \omega^\ux_\Bbbk & = d\Btheta \wedge d(\bI + Q^{\ue }_\Bbbk \cdot (X^\ue)^t + q_2 (X^{\uf})^t) \\
(\text{ with formula }\ref{equ:wedge_bold}) \;  & = \omega^\ux_\Bbbk + d [ \Btheta \cdot X^\uf ] \wedge dq_2 + \sum_{j=1}^{k_\ue } d\Btheta \cdot (X^\ue_{j,.})^t \wedge dq^\ue_j
\end{aligned}
\]

So when we sum $\zeta^* \omega^f_\Bbbk$, $\zeta^* \omega^\ue_\Bbbk$ and $\zeta^* \omega^x_\Bbbk$, we get that $\zeta^* \omega = \omega$.

Now for $\eta_B$, we can again treat separately the action on the different types $E_i$'s, and just treat the case when the $\epsilon$'s are $= -1$. We have that $E_{1,2}^* \omega_\Bbbk = E_{1,2}^* \omega^f_\Bbbk + \omega^\ue_\Bbbk + \omega^x_\Bbbk $ and:

\[\begin{aligned}
E_1^* \omega^f_\Bbbk & = d\hat{x}_1 \wedge d\hat{\xi}_1 + d\hat{x}_2 \wedge d\hat{\xi}_2 + d\hat{\theta}_3 \wedge d\hat{\xi}_3 \\ 
\ & = d\xi_1 \wedge d(-x_1) + d\xi_2 \wedge d(-x_2) + d\theta_3 \wedge d\xi_3 = \omega^f_\Bbbk
\end{aligned}
\]
\[\begin{aligned}
E_2^* \omega^f_\Bbbk & = d\tilde{x}_1 \wedge d\tilde{\xi}_1 + d\tilde{x}_2 \wedge d\tilde{\xi}_2 + d\tilde{\theta}_3 \wedge d\tilde{\xi}_3 \\
\ & = dx_2 \wedge d\xi_2 + dx_1 \wedge d\xi_1 + d\theta_3 \wedge d\xi_3 = \omega^f_\Bbbk .
\end{aligned}
\]  
 
Lastly, the transformation $(\Btheta,\bI) \mapsto (\Btheta \cdot (X^{\ux})^{-1},\bI \cdot(X^{\ux})^t)$ is a linear symplectomorphism with respect to the symplectic form $\omega^x_\Bbbk = \sum_{j=1}^{\kx} d\theta_j \wedge dI_j = d\Btheta \wedge d\bI$.
\end{proof}

This theorem means that \emph{the only symplectic invariants of the local model of a semi-toric critical value are the orientations of the half-spaces given by its elliptic components and its jet in $q_1,q_2$.} Conservation of plans orientations is only natural: symplectic structure is exactly the algebraic area on specific plans. Conservation of the jet of the focus-focus value, is a specificity of semi-toric systems. It is related to the Taylor expansion of action coordinates near a semi-toric critical value (see~\cite{Wacheux-AsymptoticsofActionnearST-2014}).


\section{ Image of moment map for a semi-toric integrable system }

 Now that we have gathered enough results concerning local models, we can prove the principal result of this paper by exhibitng a local-to-global principle. 

\begin{proof}{ of Theorem}~\ref{theo:loc_FF_values}

{\bf Local proof of 1. , 2. and 3.:}

Let $p$ be a critical point of Williamson type $\Bbbk$ with $\kf =1$ of a semi-toric integrable system $(M,\omega,F)$. Applying Theorem~\ref{theo:pres_semi-toric_G}, with the correct system of local coordinates $\varphi$ in a neighborhood $\cU$ of $p$, we have a smooth function $G$ and a matrix $A \in GL_{n-1}(\ZZ)$ such that  $F \circ \varphi =(G_1(Q_\Bbbk),A \circ \check{Q_\Bbbk})$. So, the surface $\Gamma_{\cU}:= V_\Bbbk (\cU)$ is parametrized as follows: let $\Bt$ be here the values of $\bI$. With Theorem~\ref{theo:strat_Williamson_M}, we know that $\Bt \in \tilde{D} \subseteq \RR^{\kx}$. We define now $h \smooth (\tilde{D} \to \RR)$ as $\tilde{h}(\Bt):= G_1(0,0,\underbrace{0,\ldots,0}_{k_\ue},\Bt)$. With notations of Theorem~\ref{theo:pres_semi-toric_G}, we set 

\[ \Gamma_{\cU} := V_\Bbbk (\cU) =  \left\{ \tilde{H} (\Bt) = (\tilde{h}(\Bt),\underbrace{F^\ux \circ \Bt, E^\ux \circ \Bt, X^\ux \circ \Bt}_{=: S}) | \Bt \in D \right\} . \]

On $P_\Bbbk(\cU)$, $F$ is of rank $\kx$ by definition, and in Theorem~\ref{theo:pres_semi-toric_G}, first column of $Jac(G)$ is $(\partial_{q_1}G_1,0,..,0)^T$, so there exists a linear map $S$ of $\RR^n$ such that $S \circ \tilde{\Bt} (\Bt) = (S_1 \circ \tilde{h}(\Bt),\underbrace{0,..,0}_{k_e+1},\Bt))$. 

This proves points $1.$,$2.$ and $3.$ locally: $\cP(\Gamma_\cU)$ is the affine space generated by $Im(h)$ and $Im(T)$, and with $h:= S_1 \circ \tilde{h}$ on $\cU$, we have that $\Gamma_\cU$ is the graph of $h$ on $D$. Note that since $\partial_{q_1}G_1 (0,..,0,\Bt) \neq 0$, $Im(h)$ is not in $Im(T)$.

{\bf Global proof of 1. , 2. and 3.:}

If we call $\cU_p$ the open set given for a point $p \in M$ by Theorem~\ref{theo:A-L_with_sing}, the family $\{\cU_p\}_{p \in M}$ is an open covering of $M$, so we can extract a finite one of it. If we now fix $\Bbbk$, each open set $\cU^\Bbbk_i$ gives a surface $\Gamma_i$, in $V_\Bbbk(M)$ there is at most a finite number $m_\Bbbk$ of surfaces: $V_\Bbbk(M)  = \bigcup_{i=1}^{m_\Bbbk} V_\Bbbk(\cU^\Bbbk_i) = \bigcup_{i=1}^{m_\Bbbk} \Gamma_i$. We want to show that the covering $\left\{ \cU^\Bbbk_i \right\}_{i=1..m_\Bbbk}$ of $P_\Bbbk(M)$ can be optimized in the sense that we can take a $\cU^\Bbbk_i$ containing a connected component of $P_\Bbbk(M)$. A consequence is that $m_\Bbbk$ is then minimal.

Two open sets $\cU_i$ and $\cU_j$ may intersectect. In this case, with Theorems~\ref{theo:pres_trans_B} and~\ref{theo:sympl_inv_semi-toric} we can always modify one of the two local models so that we have a local model on the union $\cU_i \cup \cU_j$ that is a natural extension of each local model. We will see that it extends also results of Theorem~\ref{theo:loc_FF_values}. 

On each open set we can apply the results proved before, and get a $h_i$ and $\cP(\Gamma_i)$ for each $\Gamma_i$. Since the change of local model between $\cU_i$ and $\cU_j$ is the identity on $P_\Bbbk(\cU_i \cap \cU_j)$, we have first that $\cP(\Gamma_i) = \cP(\Gamma_j)$. We can also set a function $h$ on the union as following 

\[ h_{\cU_i \cup \  \cU_j }:= \begin{cases}
                       h_i \text{ on } \cU_i \\
                       h_j \text{ on } \cU_j
                      \end{cases}. 
\]

It is consistent because of Theorems~\ref{theo:pres_trans_B} and~\ref{theo:sympl_inv_semi-toric}. From this, we can extend step by step Theorem~\ref{theo:loc_FF_values} from a $\cU^\Bbbk_i$ containing a $p$ to the reunion of all $\cU^\Bbbk_j$ path connected to $p$. It is an open set, the finite union of $\cU^\Bbbk_i$ that contains the connected component of $P_\Bbbk (M)$ containing $p$. This family of open sets is finite and disjoint. Thus, $P_\Bbbk (M)$ has a finite number of connected components which are strongly separated.

This proves $1.$, $2.$ and $3.$ globally. 

{\bf Proof of Item 4.:}

If we suppose that the fibers are connected, then for a critical value $v$, it now makes sense to talk of its Williamson index $\Bbbk$. Taking $p \in F^{-1}(v)$, there is a unique connected component of $P_\Bbbk (M)$ that contains $p$, and hence a unique $\Gamma_i$ in $V_\Bbbk(M)$ that contains $v$. The connected components of $P_\Bbbk (M)$ being strongly isolated, we have that $\Gamma_i$ is strongly isolated as well.

\end{proof}

\section{Conclusion} 

In this article, we have presented local techniques for the investigation of semi-toric and almost-toric systems. In the description of the image of the moment maps, one shall think about the $Q_\Bbbk$ as ``singular local coordinates'' for the image of the moment map. We have showed how these singular local coordinates can be used to describe ``singular manifolds'', by analogy with local coordinates and manifolds. 

In the case $2n = 6$, Theorem~\ref{theo:loc_FF_values} is very visual: $FF-X$ critical values are a union of nodal paths; each nodal path is contained in a plane. All these planes share a common direction, and a nodal path is the embedding of the graph of a smooth function from $\RR$ to $\RR$.

As calculus in local coordinates for regular manifolds, the calculus in local singular coordinates is a very efficient techniques to provide results even in unfriendly settings. For instance, although we don't know the precise form it may take, there must be an extension of Theorem~\ref{theo:loc_FF_values} to almost-toric systems of any complexity, or to non-compact manifolds.

Local models are one technique, it is not the only one. In upcoming articles, we shall give another approach for the description of the $V_\Bbbk$, using more general arguments like Atiyah - Guillemin \& Steinberg theorem. We shall rely on it to prove that the fibers of semi-toric systems are connected.




\end{document}